\documentclass[12pt]{amsart}
\usepackage{pdfsync}
\usepackage{graphicx}
\usepackage{amssymb}
\usepackage{amsmath}
\usepackage{amsthm}
\usepackage{enumerate}
\usepackage{subcaption}
\usepackage{esvect}
\usepackage[left=3cm,top=3.8cm,right=3cm]{geometry} % for a good layout
\usepackage{ucs}
\usepackage{cite}
\usepackage{amsxtra}
\usepackage{epstopdf}
\usepackage{bbold}
\usepackage[utf8x]{inputenc}
\usepackage{verbatim}                                  %   \begin{comment}  \end{comment}
\usepackage[toc,page]{appendix}                        %  \appendix
\usepackage{pdfsync} % pdf inverse search

\usepackage{xcolor}
\usepackage[pagebackref]{hyperref}
\hypersetup{
    colorlinks,
    linkcolor={red!60!black},
    citecolor={blue!60!black},
    urlcolor={blue!90!black}
}

\DeclareMathOperator{\diam}{diam}
\DeclareMathOperator{\rad}{rad}
\DeclareMathOperator{\dist}{dist}

\newtheorem{theorem}{Theorem}
\newtheorem{definition}[theorem]{Definition}
\newtheorem{proposition}[theorem]{Proposition}
\newtheorem{corollary}[theorem]{Corollary}

\newtheorem{lemma}[theorem]{Lemma}

\theoremstyle{remark}

\newtheorem{remark}[theorem]{Remark}

\newcommand{\R}{\mathbb{R}}

\def\N{{\mathbb N}}

\makeatletter
\def\moverlay{\mathpalette\mov@rlay}
\def\mov@rlay#1#2{\leavevmode\vtop{%
   \baselineskip\z@skip \lineskiplimit-\maxdimen
   \ialign{\hfil$\m@th#1##$\hfil\cr#2\crcr}}}
\newcommand{\charfusion}[3][\mathord]{
    #1{\ifx#1\mathop\vphantom{#2}\fi
        \mathpalette\mov@rlay{#2\cr#3}
      }
    \ifx#1\mathop\expandafter\displaylimits\fi}
\makeatother

\begin{document}
\thispagestyle{empty}

\title{Thickness and a gap lemma in $\mathbb{R}^d$}

\author{Alexia Yavicoli}
\address{Department of Mathematics, the University of British Columbia. 1984 Mathematics Road, Vancouver BC V6T 1Z2, Canada}
\email{yavicoli@math.ubc.ca, alexia.yavicoli@gmail.com}

\thanks{Part of this work was completed while I was visiting the Hausdorff Research Institute for Mathematics (HIM). I am grateful to HIM for financial support and a productive work environment.}

\begin{abstract}
We give a definition of thickness in $\R^d$ that is useful even for totally disconnected sets, and prove a Gap Lemma type result.
We also guarantee an interval of distances in any direction in thick compact sets, relate thick sets (for this definition of thickness) with winning sets, give a lower bound for the Hausdorff dimension of the intersection of countably many of them, a result guaranteeing the presence of large patterns, and lower bounds for the Hausdorff dimension of a set in relationship with its thickness.
\end{abstract}

\maketitle

%%%%%%%%%%%%%%%%%%%%%%%%%%%%%%%%%%           %%%%%%%%%%%%%%%%%%%%%%%%%%%%%%%
%%%%%%%%%%%%%%%%%%%%%%%%%%%%%%%%%%   INTRO   %%%%%%%%%%%%%%%%%%%%%%%%%%%%%%%
%%%%%%%%%%%%%%%%%%%%%%%%%%%%%%%%%%           %%%%%%%%%%%%%%%%%%%%%%%%%%%%%%%

\section{Introduction and main results}

\subsection{Historical background}

In 1970, S.~Newhouse \cite{Newhouse} introduced the notion of thickness for Cantor sets in the real line, motivated by the study of homoclinic bifurcations in dynamical systems, where it is important to know whether two Cantor sets intersect robustly.  Note that any compact set $C$ in the real line can be constructed by starting with a closed interval and removing a sequence of disjoint open intervals (called gaps) in order of decreasing length. Each gap $G_n$ is removed from a closed interval $I_n$, leaving in the next step of construction two closed intervals $L_n$ and $R_n$: the left and right pieces of $I_n \setminus G_n$. Newhouse defined the thickness of $C \subset \R$ as
\[
\tau_{\mathcal{N}} (C):= \inf_{n} \frac{\min \{ |L_n| , |R_n| \}}{|G_n|},
\]
where we denote by $|I|$ the length of the interval $I$.

Newhouse's crucial Gap Lemma \cite{Newhouse} states that given two compact sets $C^1, C^2 \subseteq \R$, such that neither set lies in a gap of the other and $\tau_{\mathcal{N}}(C^1) \tau_{\mathcal{N}}(C^2) > 1$, then
\[
C^1 \cap C^2 \neq \emptyset.
\]
Among dynamically defined Cantor sets, thickness is a continuous function of the dynamics generating the set, so that this intersection is in fact robust in the defining dynamics.

Thickness, the Gap Lemma and its generalizations turned out to be a useful notion also beyond dynamics. The structure of the intersection of thick sets was studied in
\cite{HKY, Williams}. S.~Astels \cite{Astels} investigated the thickness of sumsets, applying these results to sets arising in diophantine approximation. Recent developments around thickness include applications to problems in geometric measure theory  \cite{ST20, MT21}, number theory \cite{Y20} and the existence of patterns in fractals \cite{MT21, Y21}. However, the definition of thickness depends crucially on the order structure of the real numbers and so it remains a challenging problem to extend it in a satisfactory way to higher dimensions. The goal of this paper is to introduce a notion of thickness in arbitrary dimension that has some of the desirable properties of the one dimensional one, it reduces to Newhouse thickness in dimension $1$, and has a corresponding Gap Lemma.

Before introducing our definition, we comment on the notions of thickness in higher dimensions that have been proposed. S.~Biebler \cite{Biebler} introduced a notion of thickness for dynamically defined subsets of the complex plane, and proved a corresponding gap lemma. The main drawback is that it is restricted to this special class of sets in the plane, and even then the gap lemma is satisfied only for some ``densely packed'' examples. D-J.~Feng and Y-F.~Wu \cite{DJF} defined a notion of thickness for general subsets of $\R^d$. However, they were motivated by a problem on iterated sumsets (that S.~Astels \cite{Astels} solved in the line using Newhouse thickness) rather than analogs of the Gap Lemma. Together with K. Falconer \cite{FalconerYavicoli}, we introduced a notion of thickness in $\mathbb{R}^d$ that agrees with Newhouse's definition in the line and, among other results, proved a Gap Lemma for it. However, our definition is based on the components of the complement of the set, and so in dimension $d\ge 2$, totally disconnected sets have thickness zero (and so none of our results apply to them).

While the concept of thickness in this paper is partly inspired by Biebler's, it holds for general compact sets in any dimension, and even in Biebler's settings our Gap Lemma is more general and allows for less densely packed examples. Unlike the thickness from \cite{FalconerYavicoli}, our notion covers many totally disconnected sets. And unlike  \cite{DJF}, our motivation is to provide a Gap Lemma that allows to prove that many pairs of compact sets intersect robustly.

\subsection{Definition of thickness and main results}

From now on, we will work on $\R^d$ equipped with a distance $\dist$ associated to a norm, and all balls will be closed. The radius of a ball $B$ is denoted by $\rad(B)$. We denote by $aB$ the ball with the same center as $B$ and radius $a\rad(B)$.
The most natural norm to consider for us is the $\ell^\infty$ norm $\|(x_1,\ldots,x_d)\|_\infty = \max_i |x_i|$ due to the tiling properties of cubes, but for the most part we work in this greater generality.

Given a word $I$ (a finite sequence of natural numbers) we denote by $\ell(I) \in \N_0$ the length of $I$. We consider any compact sets that can be written as
\[
C=\bigcap_{n \in \N_0} \bigcup_{\ell(I)=n}S_I
,\]
where
\begin{itemize}
\item Each $S_I$ is a ball (in the distance $\dist$) and contains $\{ S_{I,j} \}_{1 \leq j \leq k_I}$. No assumptions are made on the separation of the $S_{I,j}$.
\item For every infinite word $i_1, i_2, \cdots$ of indices of the construction, \[\lim_{n \to +\infty} \rad S_{i_1,i_2, \cdots, i_n}=0.\]
\end{itemize}
In this case we say that $\{S_I\}_I$ is a \emph{system of balls} for $C$. Note that any compact set arises from a system of balls (for example, the closed dyadic cubes intersecting the set). In fact, there are multiple systems of balls for a given set (a system of balls is roughly equivalent to a \emph{derivation} in the one-dimensional case).

Let us fix a system of balls. The main issue with extending the notion of thickness to Cantor sets in $\R^d$ is that there isn't a suitable notion of ``gap''. We use the following notion as a substitute:
\[
h_I:= \max_{x \in S_I} \dist (x, C).
\]
Observe that if $z \in S_I$ and $r \in [h_I, +\infty)$, then the closed ball with center $z$ and radius $r$ intersects $C$.
On the other hand, for every $r <h_I$ there is a $z \in S_I$ so that the closed cube with center $z$ and radius $r$ does not intersect $C$.

We remark that $\dist(x,C)$ in the definition of $h_I$ may be realized by a point in $C\setminus S_I$. However, it is easy to see that if $x\in S_I$, then $\dist(x, C\cap S_I)\le 2 h_I$: apply the definition of $h_I$ to a point $y$ such that $\dist(x,y)\le h_I$ and $\dist(y,\partial S_I)\ge h_I$.

\begin{definition}[Thickness of $C$ associated to the system of balls $\{S_I\}_I$]
\[\tau(C, \{S_I\}_I):= \inf_{n \in \N_0} \inf_{\ell(I)=n} \frac{\min_i \rad(S_{I,i})}{h_I}.\]
\end{definition}

In order to state our gap lemma, we need an additional definition; it is related to Biebler's notion of ``well-balanced'' \cite{Biebler}.
\begin{definition}
Given $\{S_I\}_I$ a system of balls for the compact set $C$, we say that $\{S_I\}_I$ is $r$-uniformly dense if for every $I$, for every ball $B \subseteq S_I$ with $\rad (B)\geq r\rad (S_I)$, there is a child $S_{I,i} \subseteq B$
\end{definition}

\begin{theorem}[Gap Lemma]\label{Theo:GapLemma}
Let $C^1$ and $C^2$ be two compact sets in $(\R^d, \dist)$, generated by systems of balls $\{S_I^1\}_I$ and $\{S_L^2\}_L$ respectively, and fix $r \in (0, \frac{1}{2})$. Assume:
\begin{enumerate}[\rm(i)]
\item\label{hypTau} $\tau(C^1, \{S_I^1\}_I) \tau(C^2, \{S_L^2\}_L)\ge \frac{1}{(1-2r)^2}$,
\item\label{hyp0} $C^1 \cap (1-2r) S_{\emptyset}^2 \neq \emptyset$,
\item\label{hyp0b} $\rad (S_{\emptyset}^1) \geq r \rad (S_{\emptyset}^2)$ and $\rad (S_{\emptyset}^2) \geq r \rad (S_{\emptyset}^1)$,
\item\label{hypUnif} $\{S_I^1\}_I$ and $\{S_L^2\}_L$ are $r$-uniformly dense.
\end{enumerate}
Then, $C^1 \cap C^2 \neq \emptyset$.
\end{theorem}

We make some remarks on this statement.
\begin{enumerate}
\item Unlike Newhouses's Gap Lemma, we need the additional ``uniform denseness'' assumption. It is necessary to deal with systems of balls giving ``artificially large'' thickness. For example, the point $\{0\}$ can be obtained from the system of chunks $\{[-a_n, a_n]^d\}$ where $a_n$ is any sequence decreasing to $0$ (here each ball has a unique child). It is easy to check that the thickness can be arbitrarily close to $1$ if $a_n$ has slow decay. On the other hand, one can easily construct sets $C$ with arbitrarily large thickness with the same initial set $S_\emptyset = [-a_0,a_0]^d$, and such that $0\notin C$. Of course, the system of chunks for $\{0\}$ is not $\tfrac{1}{2}$-dense, so this does not contradict Theorem \ref{Theo:GapLemma}.
\item There is a balance between the thickness and the denseness assumptions: it is enough that the product of the thicknesses is $>1+c$, at the cost of requiring the sets to be dense in a way that depends on $c$. The numerical relation between both parameters is unlikely to be sharp.
\item The second and third conditions are quite mild; they can be seen as a more quantitative version of the hypothesis that each Cantor set is not contained in a gap of the other in the original Gap Lemma.
\item All the hypotheses are robust under various types of perturbations of the Cantor sets. This is discussed in more detail in Section \ref{sec:examples}.
\end{enumerate}

\section{Basic properties and examples}
\label{sec:basic}

In this section we discuss some basic properties of our notion of thickness and provide the first examples of its calculation. We show that it coincides with Newhouse thickness in the line, compare it to Biebler's thickness in the plane, and discuss its relationship to Hausdorff dimension.

\subsection{The one-dimensional case}
\label{subsec:one-dim}

In the line, our notion coincides with Newhouse's thickness using the most natural system of balls. Note that the choice of (norm-based) distance plays no role in the real line as balls are always closed intervals whose length is proportional to the radius.

Given a compact set $C \subseteq \R$, let $(G_n)_n$ be its sequence of gaps, in order of decreasing diameter, with corresponding left and right intervals $L_n, R_n$. Let us consider the system of balls defined as follows: the root is the convex hull of $C$. We remove $G_1$ from it to obtain the children $L_1, R_1$. We continue inductively: for each interval of step $k$, we remove the first gap $G_m$ in the sequence that is contained in it, to obtain the two children $L_n, R_n$.

By construction, each interval has exactly 2 children. If we denote by $S_I$ the interval from which $G_n$ is removed, then $S_I=L_n \cup G_n \cup R_n$, and
\[
\frac{\min \{|L_n|, |R_n|\}}{|G_n|}= \frac{\min_i \diam (S_{Ii})}{2 h_I}=\frac{\min_i \rad (S_{Ii})}{h_I}.
\]
Since every gap is removed at some point, we conclude that
\[
\tau_{\mathcal{N}}(C)=\tau(C, \{S_{I}\}_I).
\]

\subsection{A class of self-similar examples}
\label{subsec:simple-example}

As a first class of examples, we consider equally-spaced self-similar sets in $\R^d$. More general self-similar sets are studied in Section \ref{sec:examples}.

Fix $n \ge 2$ and $\ell \in (0, \frac{2}{n})$. We consider a Cantor set $C = C_{\ell,n} \subseteq (\R^d, \dist_{\infty})$ defined as follows:
We start with $S_\emptyset =[-1,1]^d=B[0,1]$. We take $g=\frac{2-n\ell}{n-1}$, so that $n\ell +(n-1)g=2$.
We construct the set inductively, replacing each cube from the previous step by $n^d$ equidistant cubes of relative side-length $\frac{\ell}{2}$ (in such a way that the corners of cubes in each step remain corners of cubes in the next step).  This set is the attractor of a particular type of IFS with $n^d$ homothetic functions with contraction factor $\frac{\ell}{2}$, so that the cubes are equidistant. See Figure \ref{fig:selfsim}. We consider the most natural system of balls for this example, given by the cubes in the construction.

\begin{figure}
  \centering
  \includegraphics[width=0.6\textwidth]{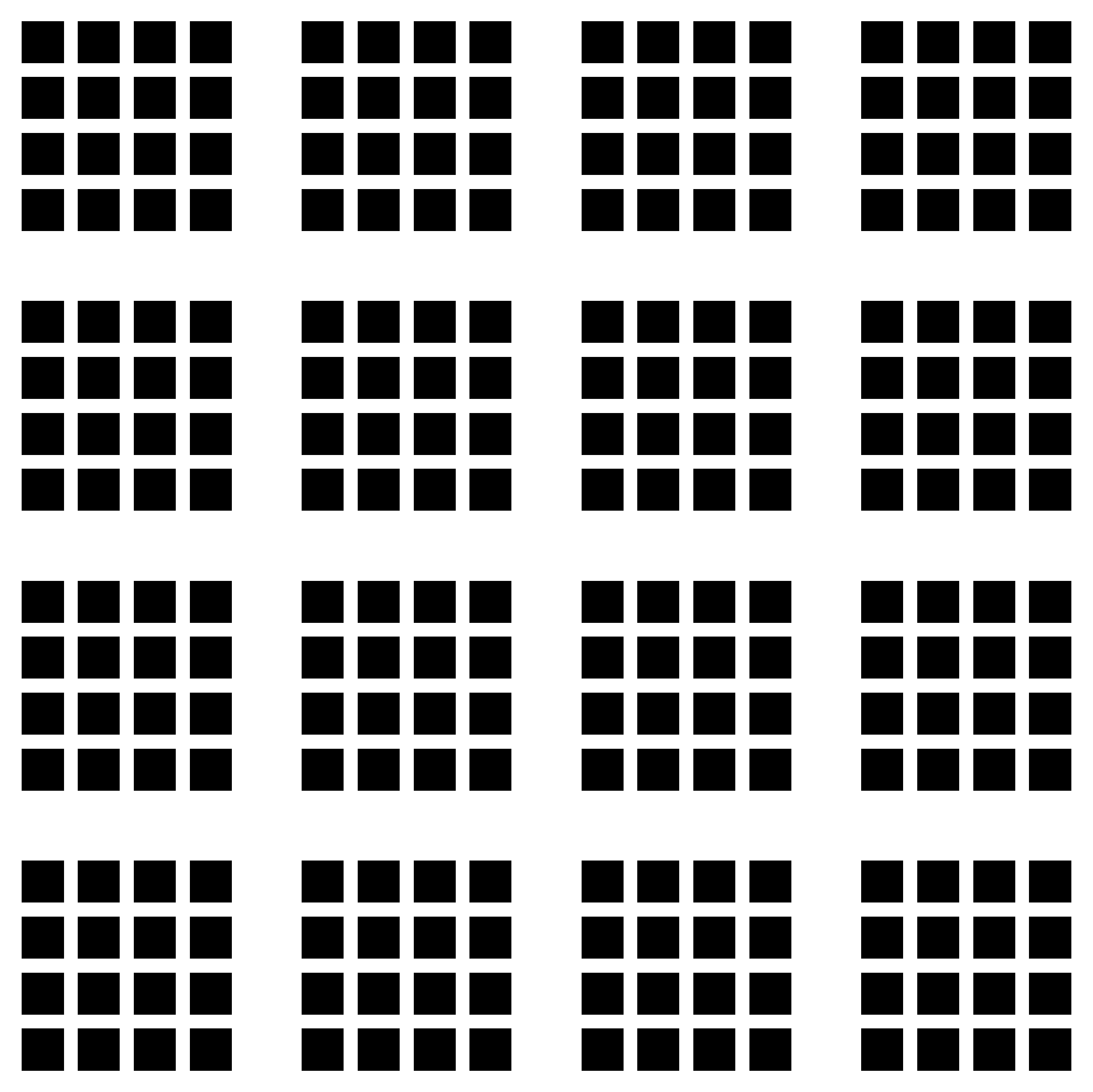}
  \caption{The second step of construction of a self-similar set with $n=4$, $\ell=2/5$ and $\tau=3$}\label{fig:selfsim}
\end{figure}

We denote any cube of step $k$ by $S_k$. By construction, cubes of the same level have the same radius and  $\rad (S_k)=\frac{\ell}{2}\rad(S_{k-1})$. We claim that
\[
h_{S_{k-1}}=\frac{g}{2}\rad(S_{k-1}).
\]
Indeed, let's consider $x \in S_{k-1}$.
If $x \notin C$ does not belong to any cube of step $k$ of the construction, then $B_{\infty}[x, \frac{g}{2}\rad(S_{k-1})]$ intersects $C \cap \partial Q$, where $Q$ is the closest cube of step $k$ to $x$.
In case $x \notin C$ is contained in a cube $S_k$ of step $k$, there is $n \geq k$ so that $x \in S_{n-1}$ but $x$ does not belong to any cube of level $n$. As before, $B_{\infty}[x, \frac{g}{2} \rad{S_{n-1}}]$ intersects $C \cap \partial Q$ where $Q$ is a cube of level $n$ contained in $S_{n-1}$. Then, the distance from $C$ to $x$ is at most $\frac{g}{2}\rad(S_{n-1}) \leq \frac{g}{2}\rad(S_{k-1})$.
So, $h_{S_k}\leq \frac{g}{2}\rad(S_{k-1})$.
And $h_{S_k}\geq \frac{g}{2}\rad(S_{k-1})$ can be seen taking $2^d$ sub-cubes of step $k$ of the construction so that the projection to each axis is formed by two consecutive intervals, and consider the midpoint of all of them in the definition of $h_{S_k}$.

It follows from the claim that the thickness is given by
\begin{equation} \label{eq:tau-C}
\tau(C) =\frac{\ell}{g}=\frac{\ell (n-1)}{2-n\ell}.
\end{equation}

One can also see that it is $r:=\frac{1}{2}\left( 2 \ell + \frac{2-n\ell}{n-1}\right)= \ell +\frac{2-n\ell}{2(n-1)}$- uniformly dense. This holds since in each direction ($\R$) any interval of length $2r=2\ell +\tilde{r}$ contained in $[-1,1]$, must contain at least one of the $\tilde{r}$-equispaced intervals of length $\ell$.

This class provides our first non-trivial examples to which Theorem \ref{Theo:GapLemma} applies. Note that $C_{l_1,n_1}$ and $C_{l_2,n_2}$ always intersect at the corners of $[-1,1]^d$. But Theorem \ref{Theo:GapLemma} guarantees that, under easily checkable assumptions on $l_i$ and $n_i$, the intersection of $C_{l_1,n_1}$ and a small translate of $C_{l_2,n_2}$ is nonempty - and hence $C_{l_1,n_1}-C_{l_2,n_2}$ contains an interval. In Section \ref{sec:examples}, we will expand both the class of self-similar sets and the type of perturbation applied.

\subsection{Comparison with Biebler's thickness and Gap Lemma}
\label{subsec:Biebler}

Biebler's definition \cite[Section 2]{Biebler} of thickness for dynamically defined sets in the complex plane is quite involved.  We can compare it to ours for the family of corner self-similar sets from \S \ref{subsec:simple-example} (which in the plane falls under the class of sets for which the thickness from \cite{Biebler} is defined). An inspection of the definition shows that
\[
\tau_{\mathcal{B}}(C_{\ell,n}) =\frac{\frac{\ell}{2}}{\sqrt{\sqrt{2}g}},
\]
where $g=\frac{2-n\ell}{n-1}$ is the gap as in \S \ref{subsec:simple-example}. From \eqref{eq:tau-C}, see that
\[
\frac{\tau(C_{\ell,n})}{\tau_{\mathcal{B}}(C_{\ell,n})}= \frac{\ell}{g}   \frac{\sqrt{\sqrt{2}g}}{\frac{\ell}{2}} = \frac{2^\frac{5}{4}}{\sqrt{g}}.
\]
In particular,  since $g<\frac{2}{n-1}$,
\[
\tau > 2^\frac{3}{4} \sqrt{n-1} \tau_{\mathcal{B}}.
\]

Moreover, in the self-similar case, Biebler requires a $(1/20)$-uniform denseness condition for the Euclidean distance (in the general conformal case the value of $r$ is smaller). So both our hypotheses in the Gap Lemma hold more broadly (even though Biebler only requires the product of the thicknesses to exceed $1$). Moreover, unlike Biebler, in our Gap Lemma we allow the sets $S_{\emptyset}^1$ and $S_{\emptyset}^2$ to be different.

Finally, we compare the proof strategies. Biebler constructs a sequence of nested chunks of the same Cantor set that intersect the other Cantor set. While we inductively construct a sequence of chunks in one of the Cantor sets that intersects the other Cantor set, the roles of the Cantor sets can change at any step of the induction. This results in a simpler and more flexible argument.

\subsection{Relationship with Hausdorff dimension}
\label{subsec:Hausdorff-dim}

It is well known that, in the real line, sets of large thickness have large Hausdorff dimension, see \cite[Proposition 5]{PalisTakens}, while the converse is easily seen to be false. Indeed, thickness can be seen as a more robust and uniform notion of size compared to Hausdorff dimension, which helps explain its usefulness. In higher dimension, we have:

\begin{lemma} \label{lem:Hausdorff-dim}
Let $C \subseteq \R^d$ be a compact set with positive thickness $\tau:=\tau (C, \{S_I\}_I)$, where each ball has at least $M_0$ children.
Assume that there is a constant $c \in (0,2)$ so that for any $S$ in the system of balls, and for any two (distinct) children $S_1$ and $S_2$ of $S$ we have
\begin{equation} \label{eq:separation}
\dist (S_1, S_2) \geq c \rad (S).
\end{equation}
Then,
\[
\dim_H (C) \geq \frac{d}{1+\frac{\log (1+\frac{1}{\tau})}{\log (M_0)}}.
\]
\end{lemma}

\begin{proof}
By scaling, we can assume that $r_{\emptyset}=\rad (S_{\emptyset})=1$. Let $r_I$ denote the radius of $S_I$. By the definition of thickness, we know that $h_I \leq r_{Ij}/\tau$ for all $I,j$.

Let $M_I$ be the number of children of $S_I$. We denote the closed $\delta$-neighbourhood of a set $A$ by $N(A, \delta)$. By the definition of $h_I$,
\[
S_I \subseteq \bigcup_{1\leq j\leq M_I} N(S_{Ij}, h_I).
\]
Then, using that $\mathcal{L}^d(B_r)=c\cdot r^d$ where $c>0$ depends on $d$ and the norm,
\begin{align*}
c \cdot r_I^d &= \mathcal{L}^d(S_I) \leq \mathcal{L}^d \left( \bigcup_{1\leq j\leq M_I} N(S_{Ij}, h_I)\right) \\
&= c \cdot \sum_{1\leq j\leq M_I} (r_{Ij}+h_I)^d \leq c \cdot \sum_{1\leq j\leq M_I} r_{Ij}^d (1+\tau^{-1})^d.
\end{align*}
So, for every $I$ we get
\[
\frac{1}{(1+\tau^{-1})^d}\leq \sum_{1\leq j\leq M_I} \left( \frac{r_{Ij}}{r_I}\right)^d.
\]
One can check that, for $\beta, \sigma>0$, $M \in \mathbb{N}$, the minimum of
\[
\left\{\sum_{1 \leq j \leq M} x_j^{d\beta}: \ x_j>0, \ \sum_j x_j^d=\sigma\right\}
\]
is attained when all $x_i$ are equal. Applying this to $M=M_I$, $x_j= \frac{r_{Ij}}{r_I}$, $\sigma=\frac{1}{(1+\frac{1}{\tau})^d}$, and
\[
\beta_I=\frac{\log (M_I)}{\log(M_I)+\log(1+\frac{1}{\tau})}= \frac{1}{1+\frac{\log(1+\frac{1}{\tau})}{\log (M_I)}}\geq \frac{1}{1+\frac{\log(1+\frac{1}{\tau})}{\log (M_0)}}=:\beta,
\]
we see that
\[
\sum_{1\leq j\leq M_I} \left( \frac{r_{Ij}}{r_I}\right)^{d\beta} \ge \sum_{1\leq j\leq M_I} \left( \frac{r_{Ij}}{r_I}\right)^{d\beta_I}  \ge 1.
\]

Hence for each $I$ there is $\widetilde{\beta}_I \ge \beta_{I}\ge \beta$ such that $\sum_{1\leq j\leq M_I} \left( \frac{r_{Ij}}{r_I}\right)^{d\widetilde{\beta}_I}  = 1$. We define a probability measure $\mu$ on $C$ by setting
\[
\mu (S_{u_1 \cdots u_k}):=\left( \frac{r_{u_1 \cdots u_k}}{r_{u_1 \cdots u_{k-1}}}\right)^{d\widetilde{\beta}_{r_{u_1 \cdots u_{k-1}}}}
\left( \frac{r_{u_1 \cdots u_{k-1}}}{r_{u_1 \cdots u_{k-2}}}\right)^{d\widetilde{\beta}_{r_{u_1 \cdots u_{k-2}}}} \cdots \left( \frac{r_{u_1}}{r_{\emptyset}}\right)^{d\widetilde{\beta}_{r_{\emptyset}}}.
\]
Since $\beta \leq \beta_I$ for all $I$ and $r_{\emptyset}=1$, we get
\[
\mu (S_{u_1 \cdots u_k})\leq \left( \frac{r_{u_1 \cdots u_k}}{r_{u_1 \cdots u_{k-1}}}\right)^{d\beta}
\left( \frac{r_{u_1 \cdots u_{k-1}}}{r_{u_1 \cdots u_{k-2}}}\right)^{d\beta} \cdots \left( \frac{r_{u_1}}{r_{\emptyset}}\right)^{d\beta}= r_{u_1 \cdots u_k}^{d\beta}.
\]

Fix a ball $B(x,r)$. If the ball intersects at most one $S_k$ ball for every $k$, then $\mu(B(x,r))=0$. Otherwise, there is a level $k\ge 0$ so that $B(x,r)$ intersects just one ball $S_k$ of level $k$, and it intersects at least $2$ children of $S_k$. The diameter $2r$ is at least the distance between these children, so we get from the separation assumption \eqref{eq:separation} that $2r \geq c r_{S_k}$. Hence, since $B(x,r)$ intersects just $S_k$, we get
\[
\mu (B(x,r)) \leq r_{S_k}^{d\beta} \leq (2/c)^{d\beta} r^{d\beta}.
\]
The claim follows from the mass distribution principle.
\end{proof}

\begin{remark}
In the real line, this bound is similar but worse compared to the one in \cite{PalisTakens}.  In their proof they exploit the order structure of the real numbers (they use that there exists of a system of chunks so that each chunk has exactly $M_0=2$ children which are located at the left and right ends of the chunk), so it seems difficult to extend the same argument to higher dimensions, where the $h_I$ can be realized at the boundary of the corresponding ball.
\end{remark}

\section{Proof of the Gap Lemma}

\label{sec:gap-lemma}

In this section, we prove Theorem \ref{Theo:GapLemma}. We start by noting a simple consequence of the definitions.

\begin{remark}\label{remark:h}
If a system of balls $\{S_I\}_I$ is $r$-uniformly dense with respect to $C$, then
\[
\max_i \rad{S_{Ii} \leq r \rad (S_I)} \ \text{ and } \  h_I \leq 2r \rad(S_I).
\]
The second inequality follows from the fact that for each $x\in S_I$ we can find a ball $B$ of radius $r$ with $x\in B\subset S_I$ (if $x$ is close to the boundary of $S_I$ then we can't take $x$ to be the center). We also get from the first inequality and the definition of thickness that
\[
\tau \leq r \inf_{n \geq 0} \inf_{\ell(I)=n} \frac{\rad(S_I)}{h_I}.
\]
\end{remark}

\begin{proof}[Proof of Theorem \ref{Theo:GapLemma}]
By compactness, in order to prove the theorem,it suffices to show that there is a sequence of points $(\alpha_n^1, \alpha_n^2) \in C^1 \times C^2$ so that $\dist (\alpha_n^1, \alpha_n^2) \to 0$.

To do this, it is enough to prove by induction that for every $n \geq 0$ there are $j_n \in \{1,2\}$, a word $J_n^{3-j_n}$ and
\[
\alpha_n:=\alpha_n^{j_n} \in C^{j_n} \cap (1-2r) S_{J_n}^{3-j_n},
\]
where $\rad (S_{J_{n+1}}^{3-j_{n+1}}) \leq r \rad (S_{J_{n}}^{3-j_n})$. We can then take any point $\alpha_n^{3-j_n} \in C^{3-j_n} \cap  S_{J_n}^{3-j_n}$ to get the desired sequence $(\alpha_n^1, \alpha_n^2)$.

By hypothesis \eqref{hyp0}, there is $\alpha_0=\alpha_0^1 \in C^1 \cap (1-2r) S_{\emptyset}^2$, so the case $n=0$ holds.

We assume the claim holds for some $n$ and will establish it for $n+1$. Other than \eqref{hyp0}, which will not be used in the rest of the proof, the assumptions and  the inductive claim are symmetric, so we can assume that $\alpha_n \in C^1\cap (1-2r) S_L^2$. We are going to prove that
there are $j_{n+1} \in \{1,2\}$ and
\[
\alpha_{n+1}:=\alpha_{n+1}^{j_{n+1}} \in C^{j_{n+1}} \cap (1-2r) S_{J_{n+1}}^{3-j_{n+1}},
\]
with $\rad(S_{J_{n+1}}^{3-j_{n+1}})\le r \rad(S_L^2)$.

Since $\alpha_n \in C^1$, there exists a sequence $\{S_k^1\}_{k \in \N_0}$, where $S_k^1$ is a ball of $C^1$ of level $k$ of the construction containing $\alpha_n$.
Suppose first that $(1-2r)\rad (S_{\emptyset}^1) \geq h_L^2$. Then  we take $k \geq 0$ to be the largest such that $(1-2r) \rad (S_k^1)\geq  h_L^2$. By hypothesis \eqref{hypTau},
\begin{align*}
\frac{1}{(1-2r)^2} &\leq \tau(C^1, \{S_I^1\}_I) \tau(C^2, \{S_L^2\}_L) \\
&\leq \frac{\rad(S_{k+1}^1)}{h_{S_k}^1} \frac{\min_{\ell} \rad (S_{L, \ell}^2)}{h_L^2}<\frac{\min_{\ell}\rad (S_{L, \ell}^2)}{(1-2r) h_{S_k^1}}.
\end{align*}
%So, $h_{S_k^1} < (1-2r)\min_{\ell} \rad (S_{L, \ell}^2)$.
Otherwise, if $(1-2r)\rad (S_{\emptyset}^1) < h_L^2$, we take $k=0$ (so $S_k^1=S_{\emptyset}^1$). By hypothesis \eqref{hypTau} we have
\begin{align*}
\frac{1}{(1-2r)^2} &\leq \tau(C^1, \{S_I^1\}_I) \tau(C^2, \{S_L^2\}_L) \\
&\leq \frac{\rad(S_{1}^1)}{h_{\emptyset}^1} \frac{\min_{\ell} \rad (S_{L, \ell}^2)}{h_L^2}<\frac{\min_{\ell}\rad (S_{L, \ell}^2)}{(1-2r) h_{\emptyset}^1}.
\end{align*}
Hence, in any case,
\begin{equation}\label{eq:smallhole}
h_{S_k^1} < (1-2r)\min_{\ell} \rad (S_{L, \ell}^2).
\end{equation}

We split the proof into two cases.

\textbf{Case 1}: $\rad (S_k^1) \geq  r \rad (S_L^2)$.

We claim that there is a ball $B$ contained in $S_k^1 \cap S_L^2$, of radius at least $r \rad (S_L^2)$.
If $S_k^1 \subseteq S_L^2$, we take $B:=S_k^1$. Otherwise, since $\alpha_n \in (1-2r)S_L^2\cap S_k$ by the inductive hypothesis and construction, we have that $(S_L^2)^C$ and $(1-2r) S_L^2$ both intersect $S_k$. Let $M$ be the line through the centers of $S_k^1$ and $S_L^2$ (if the centers are equal, the ball centered at that point of radius $r \rad(S_L^2)$ works). We let $y$, $z$ points where $M$ hits $\partial [(1-2r) S_L^2]$ and $\partial S_L^2$ respectively, with $\dist(y,z)=2r$. If $x$ is the midpoint of $y,z$, then $B=B[x,r\rad(S_L^2)]$ is the ball we are looking for.

Then, by the denseness assumption, there is a child $S_{L,\ell}^2 \subseteq B \subseteq S_k^1 \cap S_L^2$. In particular, $(1-2r)S_{L,\ell}^2 \subseteq S_k^1$. Recalling \eqref{eq:smallhole} and applying the definition of $h_{S_k}^1$ to the center of $S_{L, \ell}^2$, we get that there is $\alpha_{n+1} \in C^1\cap (1-2r)S_{L,\ell}$.
In this case we take $j_{n+1}=1$, $S_{J_{n+1}}^2=S_{L,\ell}^2$ which satisfies $\rad(S_{L,\ell}^2)\le r \rad(S_L^2)$ by Remark \ref{remark:h}.

\textbf{Case 2}: $\rad (S_k^1) <  r \rad (S_L^2)$.

Since we know that $\alpha_n \in S_k^1 \cap (1-2r) S_L^2 \neq \emptyset$, we get that $S_k^1 \subseteq S_L^2$.

Note that in this case $k \neq 0$ by hypothesis \eqref{hyp0b}, which means that $k$ is the largest so that $(1-2r) \rad (S_k^1)\geq h_L^2$.
Since $(1-2r) S_k^1$ is a ball of radius $(1-2r) \rad (S_k^1)\geq h_L^2$ that is contained in $S_L^2$, applying the definition of $h_L^2$ to the center of $(1-2r) S_k^1$, we get that there is a $\alpha_{n+1} \in (1-2r) S_k^1 \cap C^2$. In this case we take $j_{n+1}=2$, $S_{J_{n+1}}^1=S_{k}^1$ which satisfies the shrinking condition by the assumption of Case 2.

This completes the proof of Theorem \ref{Theo:GapLemma}.
\end{proof}

\section{Examples and applications}

\label{sec:examples}

\subsection{Self-homothetic sets}

We show that thickness can be explicitly bounded for self-homothetic sets in any dimension. See e.g. \cite[\S 2.2]{FalconerTechniques} for an introduction to iterated function systems (IFS's).
\begin{lemma} \label{lem:thickness-selfsim}
Let $C \subseteq (\R^d, \dist)$  be the attractor of an IFS formed by homothetic functions $f_i(x)=\lambda_i x +t_i$ with $\lambda_i \in (0,1)$. Suppose $f_i(B[0,1])\subset B[0,1]$ (which can always be achieved by a change of coordinates) and let $S_{i_1\ldots i_k} = f_{i_1}\cdots f_{i_k}(B[0,1])$.
Then,
\begin{equation} \label{eq:h-claim}
h_\emptyset \leq \max_{x \in S_{\emptyset} \setminus \bigcup_i S_i} \min_i \frac{\dist (x,t_i)-\lambda_i}{1-\lambda_i}.
\end{equation}
Moreover,
\[
\tau (C, \{S_I\}_I) \geq  \frac{\min_j \lambda_j}{h_\emptyset},
\]
and $C$ is $(2 \max_i \lambda_i + h_{\emptyset})$-uniformly dense.
\end{lemma}
We make some remarks on this statement. Firstly, the estimate in \eqref{eq:h-claim} involves only a finite number of balls, so the bounds on $\tau$ and the density are effective, providing many explicit examples (without a lattice structure as in \S\ref{subsec:simple-example}) for which the Gap Lemma applies. Secondly, the bounds are clearly continuous in the parameters $\lambda_i$ and $t_i$, so we get robust intersections in the family of self-homothetic sets. The lemma does not extend to more general similarities, since the norm is not in general invariant under rotations.

\begin{proof}[Proof of Lemma \ref{lem:thickness-selfsim}]
By self-similarity, for any $I=(i_1 \cdots i_n)$ we have $\rad(S_I)=\lambda_I= \prod_{1 \leq j \leq n}\lambda_{i_j}$,
and
\begin{equation} \label{eq:h-selfsim}
h_{S_I}\leq \max_{x \in S_I} \dist (x, C \cap S_I)\leq  \lambda_I h_{\emptyset}.
\end{equation}
(Note that in case the siblings are disjoint the last inequality becomes an equality, but in general the inequality can be strict.) This yields $\tau(C, \{S_I\}) \ge \min_j \lambda_j/h_\emptyset$.

We next see that $C$ is $(2 \max_i \lambda_i + h_{\emptyset})$-uniformly dense. By self-similarity, again it is enough to look at the first level only.
Let $B=B[x,r]$ be a closed ball contained in $S_{\emptyset}$ with radius $r\geq h_{\emptyset}+2 \max_i \lambda_i$. By definition, we know there is $y \in C$ so that $\dist (x,y)=\dist (x,C)\leq h_{\emptyset}$. There exists $i$ such that $y \in S_i$. Then, we have that $S_i \subseteq B$.

It remains to show \eqref{eq:h-claim}. Using \eqref{eq:h-selfsim} twice, we get
\begin{align*}
h_\emptyset &:=\max_{x \in S_\emptyset} \dist(x,C)\\
&= \max\big\{ \max_i \max_{x \in S_i} \dist (x,C), \max_{x \in S_{\emptyset}\setminus \bigcup_i S_i} \dist(x,C) \big\}\\
&\leq \max \big\{\max_i h_i, \max_{x \in S_{\emptyset}\setminus \bigcup_i S_i} \min_i (\dist(x, S_i)+h_i)\big\}\\
&\leq  \max \big\{\max_i \lambda_i h_\emptyset, \max_{x \in S_{\emptyset}\setminus \bigcup_i S_i} \min_i (\dist(x, t_i)-\lambda_i+\lambda_i h_{\emptyset}) \big\}
\end{align*}
Since $\max_i \lambda_i<1$, the maximum is not achieved by $\lambda_i h_\emptyset$, and so
\[
h_\emptyset \leq \max_{x \in S_{\emptyset}\setminus \bigcup_i S_i} \min_i (\dist(x, t_i)-\lambda_i+\lambda_i h_{\emptyset}).
\]
Hence there is $x_0 \in S_{\emptyset}\setminus \bigcup_i S_i$ so that for all $i$ we have $h_\emptyset \leq  \dist(x_0, t_i)-\lambda_i+\lambda_i h_{\emptyset}$.
Rearranging, we get \eqref{eq:h-claim}.

\end{proof}

\subsection{Robustness under $C^1$ perturbations}

It is well known that Newhouse thickness (unlike Hausdorff dimension) is not invariant under smooth diffeomorphisms. However, it is almost invariant for ``almost-linear'' diffeomorphisms. In this section we extend this fact to our notion of thickness. For simplicity we work with the $\ell^\infty$ norm.

\begin{lemma} \label{lem:C1-robust}
Let $C \subset B_\infty[0,1]$ be a compact set with system of balls (in the $\ell^\infty$ norm) $\{S_I\}_I$, and assume $\min_j r_{Ij} \ge \lambda r_I$ for some $\lambda>0$ and all $I,j$.

Let $f: B_{\infty}[0,1]\subset \R^d \to \text{Im}(f) \subset \R^d$ an invertible $C^1$ function with $\|Df(x)-I\|_{\infty}< \varepsilon$ for all $x \in B_{\infty}[0,1]$, and
some $\varepsilon \in (0,1)$.

Then $f(C)$ has a system of balls $R_I$ (still considering $\dist_\infty)$ such that
\[
\tau(f(C),\{R_I\}_I) \geq \frac{\tau(C,\{S_I\}_I)}{1+\tau(C,(S_I)_I)\frac{2\varepsilon}{(1+\varepsilon)\lambda}}.
\]
\end{lemma}

\begin{proof}
Recall that if $A\in\text{GL}_d(R)$, then $\|A\|_{\infty}:=\max_i \sum_j |a_{i,j}|$. Since $\|Df(x)-I\|_{\infty}< \epsilon$, for $x\in B_{\infty}[0,1]$ we have:
\begin{enumerate}[(i)]
\item  $1-\varepsilon < \|Df(x)\|_{\infty}<1+\varepsilon$.
\item \label{it:D-f-inverse} $\|I-(Df(x))^{-1}\|_{\infty}\leq \|(Df(x))^{-1}\|_{\infty} \|Df(x)-I\|_{\infty}<\|Df(x)^{-1}\|_{\infty} \varepsilon$. Then,
    \[
    \|(Df(x))^{-1}\|_{\infty} \leq \|I\|_{\infty}+ \|I-(Df(x))^{-1}\|_{\infty}<1+ \|(Df(x))^{-1}\|_{\infty} \varepsilon,
    \]
and hence
\[
\|D(f^{-1})(f(x))\|_{\infty}=\|(Df(x))^{-1}\|_{\infty}< \frac{1}{1-\varepsilon}.
\]
\end{enumerate}

Let $f=(f_1, \cdots, f_d)$ with $f_i:  B_{\infty}[0,1]\subset \R^d \to \R$. By the Mean Value Theorem, $f_i(x)-f_i(y)= \nabla f_i (\xi_i) (x-y)$ where $\xi_i$ is an intermediate point. Then,
\[
|f_i(x)-f_i(y)|\leq \|\nabla f_i (\xi_i)\|_1 \|x-y\|_\infty < (1+\varepsilon) \|x-y\|_\infty.
\]
Since this holds for all $i$, we get
\begin{equation} \label{eq:almost-linear}
\|f(x)-f(y)\|_{\infty} \leq (1+\varepsilon) \|x-y\|_\infty.
\end{equation}
Writing $S_I:=B_{\infty}[z_I,r_I]$, we deduce
\[
f(S_I) \subseteq B_{\infty}[f(z_I),(1+\varepsilon) r_I]=:R_I.
\]
This shows in particular that $(R_I)_I$ is a system of balls for the compact set $f(C)$.

With the same argument as before (but applied to $f^{-1}$ instead of $f$ and using \eqref{it:D-f-inverse} above) we get
\[
\|f^{-1}(x)-f^{-1}(y)\|_{\infty} < \frac{1}{1-\varepsilon} \|x-y\|_{\infty},
\]
and therefore $f^{-1} \left( B_{\infty}[f(z_I),r_I(1-\varepsilon)] \right) \subseteq B_{\infty}[z_I,r_I]$. Then,
\[
B_I:=B_{\infty}[f(z_I),r_I(1-\varepsilon)] \subseteq f(S_I).
\]
Let us estimate $h_{R_I}$ (for $f(C)$) in terms of $h_{S_I}$:
\begin{align*}
  h_{R_I} & \leq \max_{y \in R_I} \dist_{\infty}(y, B_I)+ \max_{z \in f(S_I)} \dist_{\infty}(z, f(C)) \\
   & \leq 2r_I \varepsilon + \max_{x \in S_I} \dist_{\infty}(f(x),f(C))\\
   & \overset{\eqref{eq:almost-linear}}{\leq} 2r_I \varepsilon + (1+\varepsilon) h_{S_I}.
\end{align*}

We conclude that, for any word $I$,
\begin{align*}
\frac{\min_j \rad(R_{Ij})}{h_{R_I}} &\geq \frac{\min_j r_{Ij}(1+\varepsilon)}{ 2r_I \varepsilon + (1+\varepsilon) h_{S_I}}\\
&\overset{r_{Ij}\ge\lambda r_I}{\geq} \frac{1}{\frac{2\varepsilon}{(1+\varepsilon)\lambda} +\frac{h_{S_I}}{\min_j r_{Ij}}}\\
&\geq \frac{\tau}{1+\frac{2\varepsilon\tau}{(1+\varepsilon)\lambda}},
\end{align*}
where $\tau=\tau(C,\{S_I\}_I)$. Taking the infimum over $I$ gives the claim.
\end{proof}

\subsection{Application to directional distance sets in $\R^d$}

\begin{definition}
For a fixed  direction $v \in S^{(d-1)}$, we say that $t$ is a distance between points of $C$ in direction $v$ if there are $e_1$ and $e_2$ in $C$ such that $e_1-e_2=tv$.
Equivalently,  $C \cap (C+tv) \neq \emptyset$.

We define $\Delta_v (C)$ as the set of all distances between points in the set $C$ in direction $v$.
\end{definition}

As an application of the Gap Lemma, we have:
\begin{corollary}
Let $C=\bigcap_{n \geq 0} \bigcup_{\ell (I)=n}S_I$ be a compact set in $\R^d$ so that there exists $r \in (0, \frac{1}{3}]$ satisfying
\begin{itemize}
\item $\tau (C, \{S_I\}_I) \geq \frac{1}{1-2r}$
\item $\{S_I\}_I$ is $r$-uniformly dense with respect to $C$
\end{itemize}
Then, there is $a>0$ (depending only on $r$ and the radius of $S_\varnothing$) such that for any direction $v \in S^{(d-1)}$ we have
\[
[0,a]\subseteq \Delta_v (C).
\]
\end{corollary}

\begin{proof}
We can assume without loss of generality that $S_{\emptyset}=B[0,1]$.

Let $v$ be any vector in $S^{(d-1)}$. We are going to show that the sets $C$ and $C+tv$ satisfy the hypothesis of the Gap Lemma (Theorem \ref{Theo:GapLemma}) for $t \in [\frac{-2r}{1-2r}, \frac{2r}{1-2r}]$. Then we will have $C \cap (C+tv) \neq \emptyset$ for any  $t \in [0, \frac{2r}{1-2r}]$, and so $[0,\frac{2r}{1-2r}] \subseteq \Delta_v (C)$.

By assumption,  $\{S_I\}_I$ is uniformly dense with respect to $C$. Since thickness is preserved by translations (translating also the system of balls), we have
\[
\tau(C, \{S_I\}_I) \tau(C+tv, \{S_I+tv\}_I) \geq \frac{1}{(1-2r)^2}.
\]
It remains to show tht $C \cap (1-2r) \left(B[0,1]+tv \right)\neq \emptyset$ (the other symmetric hypothesis is analogous). Since $r \in (0, \frac{1}{3}]$, we have that $(1-2r) \left(B[0,1]+tv \right)$ is a ball of radius at least $r$.
And since $t \in [0, \frac{2r}{1-2r}]$, we have $(1-2r) \left(B[0,1]+tv \right) \subseteq B[0,1]$. Hence, by $r$-denseness, there is a child $S_i$ of $S_{\emptyset}$, contained in $(1-2r) \left(B[0,1]+tv \right)$. In particular, $C \cap (1-2r) \left(B[0,1]+tv \right)\neq \emptyset$.
\end{proof}

We remark that the fact that the distance set has non-empty interior in the setting of this corollary follows from Lemma \ref{lem:Hausdorff-dim} and the Mattila-Sj\"{o}lin Theorem \cite{MattilaSjolin}. The point of the corollary is that there is a uniform interval of distances, containing $0$, in every direction.

\section{Thickness, winning sets, and large finite patterns}

\label{sec:winning-and-patterns}

In this section we obtain a connection between thickness and games, and use it to establish the presence of homothetic copies of finite sets in sets of large thickness. We adapt ideas from \cite{Y21, FalconerYavicoli} to the notion of thickness in this paper. We note that the geometric details differ. An important difference is that we need to deal with finite unions of spheres as the basic family of sets for the game (denoted by $\mathcal{H}$ below) instead of points as in \cite{Y21, FalconerYavicoli}.

\subsection{The potential game}

We recall the definition and basic properties of the \emph{potential game} from \cite{BFS}. Let $\mathcal{H}$ be a family of closed subsets of $\R^d$.
\begin{definition}\label{gamedef}
Given $\alpha, \beta, \rho >0$ and $c \geq 0$, Alice and Bob play the $(\alpha, \beta, c, \rho, \mathcal{H})$-game in $\R^d$ under the following rules:
\begin{itemize}
\item For each $m \in \N_{0}$ Bob plays first, and then Alice plays.
\item On the $m$-th turn, Bob plays a closed ball $B_m:=B[x_m , \rho_m ]$, satisfying $\rho_0 \geq \rho$, and $\rho_{m}\geq \beta \rho_{m-1}$ and $B_m \subseteq B_{m-1}$ for every $m \in \N$.
\item On the $m$-th turn Alice responds by choosing and erasing a finite or countably infinite collection $\mathcal{A}_m$ of sets $N(L_{i,m}, \rho_{i,m})$ with $L_{i,m} \in \mathcal{H}$ and $\rho_{i,m}>0$. Alice's collection must satisfy $\sum_{i} \rho_{i,m}^c \leq (\alpha \rho_m )^c$ if $c>0$, or $\rho_{1,m} \leq \alpha \rho_m$ if $c=0$ (in the case $c=0$ Alice can erase just one set).
\item $\lim_{m \to \infty} \rho_m =0$ (Note that this is a non-local rule for Bob. One can define the game without this rule, adding that Alice wins if $\lim_{m \to \infty} \rho_m \neq 0$. But to make the definitions simpler we added this condition as a rule for Bob.)
\end{itemize}
\end{definition}

Alice is allowed not to erase any set, or equivalently to pass her turn.

There exists a single point $x_{\infty} = \bigcap_{m \in \N_0} B_m$ called the {\it outcome of the game}. We say a set $S \subset \R^d$ is an $(\alpha, \beta, c, \rho, \mathcal{H})$-{\it winning set}, if Alice has a strategy guaranteeing that if $x_{\infty} \notin \bigcup_{m \in \N_0} \bigcup_i N(L_{i,m}, \rho_{i,m})$, then $x_{\infty} \in S$.

Note that the conditions $B_0 \supseteq B_1 \supseteq \cdots$ and $\lim_{m \to \infty} \rho_m =0$ imply $\beta < 1$.

\subsection{Good properties of the game}

The following basic properties are crucial to the applications of the potential game; see \cite{BFS} for more details.

\begin{lemma}[Countable intersection property]\label{Countable intersection property}
Let J be a countable index set, and for each $j \in J$ let $S_j$ be an $(\alpha_j , \beta, c, \rho, \mathcal{H})$-winning set, where $c>0$. Then, the set $S:= \bigcap_{j \in J} S_j$ is $(\alpha ,\beta, c, \rho, \mathcal{H})$-winning where $\alpha^c = \sum_{j \in J} \alpha_j^c$ (assuming that the series converges).
\end{lemma}
To see this, it is enough to consider the following strategy for Alice: in the turn $k$ she plays the union over $j$ of all the strategies of turn $k$.

\begin{lemma}[Monotonicity]\label{Monotonicity}
If $S$ is $(\alpha , \beta, c, \rho, \mathcal{H})$-winning and $\tilde{\alpha} \geq \alpha$, $\tilde{\beta} \geq \beta$, $\tilde{c} \geq c$, $\tilde{\rho} \geq \rho$ and $\mathcal{H} \subseteq \tilde{\mathcal{H}}$, then $S$ is $(\tilde{\alpha} , \tilde{\beta}, \tilde{c}, \tilde{\rho}, \tilde{\mathcal{H}})$-winning.
\end{lemma}
This holds because  \[\Big( \sum_i \alpha_i^{\tilde{c}}\Big)^{1/\tilde{c}} \leq \Big( \sum_i \alpha_i^{c}\Big)^{1/c} \text{ when } c\leq \tilde{c},\] so Alice can answer in the $(\tilde{\alpha} , \tilde{\beta}, \tilde{c}, \tilde{\rho}, \tilde{\mathcal{H}})$-game using her strategy to answer from the $(\alpha , \beta, c, \rho, \mathcal{H})$-game.

\begin{lemma}[Invariance under similarities]\label{Invariance under similarities}
Let $f:\R^d \to \R^d$ be a  similarity with contraction ratio $\lambda$. Then a set $S$ is $(\alpha , \beta, c, \rho, \mathcal{H})$-winning if and only if the set $f(S)$ is $(\alpha , \beta, c, \lambda \rho, f(\mathcal{H}))$-winning.
\end{lemma}
This follows by mapping Alice's strategy by $f$.

\subsection{Relationship between thickness and winning sets}

\begin{definition}
Given $M\in\mathbb{N}$, we define $\mathcal{H}_M$ to be the family of all sets that are a union of at most $M$  $(d-1)$-spheres.
\end{definition}

We now establish the key property that relates winning sets to thickness.

\begin{proposition}\label{prop:compact winning}
Let $C\subset \R^d$ be a compact set with an associated system of balls $\{S_I\}_I$ where
\begin{enumerate}[\rm(i)]
\item each ball has at most $N_0$ children,
\item for each $n \geq 1$ the sets $S_I$ with $|I|=n$ are disjoint and have the same radius $r_n$.
\item $\sup_{n \geq 0} \frac{r_{n+1}}{r_n}<1$
\end{enumerate}

Then there is $M$, depending on $N_0$, the choice of the norm, and $d$, such that the following holds. If $\tau:=\tau(C, \{S_I\}_I )>0$, then $S:=C\cup (\R^d\setminus S_{\emptyset})$ is $\big(\frac{1}{\tau }, \beta, 0, \beta \rad (S_{\emptyset}), \mathcal{H}_M\big)$-winning for every $\beta \in [\sup_{n \geq 0} \frac{r_{n+1}}{r_n},1)$.
\end{proposition}

\begin{proof}
Observe that, by the assumption that  balls of the same level are non-overlapping and of the same size, there is a number $\kappa$ (depending only on the norm and the ambient dimension $d$) so that for every ball $B$ with $\rad (B)\leq \rad (S_I)$, $B$ intersects at most $\kappa$ sets $S_I$. Set $M:=(N_0+1)\kappa$.

First, we are going to be more specific about the sets that Alice will use to respond to Bob's moves. Let $S_I = B[c_I, r_I]$. We define the sets
\[
H_I:= \partial B(c_I, r_I-\frac{h_I}{2}) \cup \bigcup_i \partial B(c_{Ii}, r_{Ii}+h_I).
\]
(Note that the number of $(d-1)$-spheres forming $H_I$ is at most $N_0 +1$.). The intuitive idea is that using neighbourhoods of the sets we just defined, Alice is able to erase the complement of the children of $S_I$ that intersect Bob's move. By the definition of $h_I$,
\begin{equation} \label{eq:S-I-cover}
S_I \setminus \bigcup_i S_{Ii} \subseteq N(H_I, h_I).
\end{equation}
The reason we also delete the neighbourhood of $\partial B(c_I, r_I-\frac{h_I}{2})$ is that a priori the distance from $x\in S_I$ to $C$ could be realized at a point outside of $S_I$, but this is not possible if $x\in S\setminus N(\partial B(c_I, r_I-\frac{h_I}{2}),h_I)$.

Now, we define the kind of sets Alice may choose to erase:
\[
A:= \bigcup_{I \in \mathcal{I}} N(H_I, \max_{I \in \mathcal{I}}h_I),
\]
where $n \geq 0$ and $\mathcal{I}$ is any set of at most $\kappa$ words of length $n$. Observe that the set $A$ is the $(\max_{I \in \mathcal{I}}h_I)$-neighborhood of $\bigcup_{I \in \mathcal{I}} H_I$, which is a union of at most $M=\kappa (N_0+1)$ spheres.

Let $(B_m)_{m=0}^\infty$ be the sequence of Bob's moves. Note that since $S^C \subseteq S_{\emptyset}$, we can assume that $\rad(B_0)\leq r_0$ (Alice passes her turn until Bob plays a ball with radius $\leq r_0$). By the rules of the game and the assumption and $\beta \in [\sup_{n \geq 0} \frac{r_{n+1}}{r_n},1)$, we have that:
\begin{equation}\label{eq:radiusr1}
r_0 \geq \rad (B_0) \geq \rho:= \beta \rad (S_{\emptyset})=\beta r_0 \geq r_1,
\end{equation}
and
\begin{equation}\label{eq:AtLeastOneBall}
\forall n \in \N_0 \text{ there is at least one ball $B_m$ with radius } \rad(B_m)\in R_n,
\end{equation}
where $R_0:=[r_1,r_0]$ and $R_n:=[r_{n+1},r_{n})$ for $n \geq 1$.

%strategy
Now, we can give a strategy for Alice: Given a move $B$ by Bob, how does Alice respond?
In case $B$ is the first ball played by Bob with radius in $R_{n}$ (for some $n \geq 0$), Alice erases $A(B):=\bigcup_{I \in \mathcal{I}} N(H_I, \max_{I \in \mathcal{I}}h_I)$, where $\mathcal{I}=\mathcal{I}(B)$ is the set at most $\kappa$ words $I$ of length $n$ so that $B \cap S_I \neq \emptyset$, if it's a legal move. In any other case Alice does not erase anything (she passes her turn).

%winning
To show that this strategy is winning, suppose that $x_{\infty} \notin \bigcup_m A_m$. Let's assume that $x_{\infty} \notin S$ for the sake of contradiction. Then $x_{\infty}\in S_{\emptyset}\setminus C$, so there exists a unique $n \geq 0$ and a word $I$  of length $n$ so that $x_{\infty} \in S_I \setminus \bigcup_i S_{Ii}$. Then we know from \eqref{eq:S-I-cover} that $x_{\infty} \in N(H_I, h_I)$. We will see that Alice erases $N(H_I, h_I)$ as a response to one of Bob's moves, contradicting that $x_\infty\notin A_m$ for all $m$.

By the assumption $\tau:=\tau(C, \{S_I\}_I )>0$, we have $r_n>0$. We also know by the rules of the potential game that $\rad (B_m) \to_{m \to \infty} 0$. Using \eqref{eq:AtLeastOneBall}, we see that for each $n\geq 0$ there is a smallest $m=m(n)$ such that $\rad (B_{m}) \in R_n$.

If $m=0$, we get from \eqref{eq:radiusr1} that $n=0$ and (using the definition of thickness) $r_{0} \geq \rad (B_0) \geq r_{1} \ge \tau h_{\emptyset}$.

If $m \geq 1$ then, by the rules of the game and the definition of thickness,
\[
\rad(B_{m}) \geq \beta \rad (B_{m -1})> \beta r_{n} \geq r_{n+1}\geq   \tau \max_{|I|=n} h_I.
\]

In any case, we proved that
\begin{equation}\label{eq:winning1}
\frac{1}{\tau} \rad (B_{m}) \geq \max_{I \in \mathcal{I}(B_{m})} h_I.
\end{equation}
This says that in the step $m$ Alice erases $A(B_{m}) \supseteq N(H_I, h_I)$, since it is a legal move. As explained above, this is a contradiction, finishing the proof.

\end{proof}

\subsection{Application: Hausdorff dimension of the intersections of thick compact sets}

As in \cite{Y21, FalconerYavicoli}, combining Proposition \ref{prop:compact winning} with the results from \cite[Section 5]{BFS}, we are able to deduce that collections of many thick sets that are located close to each other have nonempty intersection (and the intersection even has positive Hausdorff dimension). Note that the Gap Lemma guarantees that two thick sets intersect, but even for Newhouse thickness it is challenging to find checkable conditions ensuring that three or more thick sets intersect.

\begin{definition}
Given $\eta >0$ and a collection of closed sets $\mathcal{H}$ in $(\R^d, \dist)$, the measure $\mu$ is called absolutely $(\eta, \mathcal{H})$-decaying if there is $L\ge 1$ such that for every sufficiently small ball $B(x, \rho)$ centered in the topological support of $\mu$, for every $H \in \mathcal{H}$, and for every $\varepsilon >0$, we have
\[
\mu (B(x,\rho) \cap  N(H, \varepsilon \rho)) \le L\, \varepsilon^\eta \mu (B(x, \rho)).
\]
\end{definition}

Note that for $M$ fixed, the Lebesgue measure $\mathcal{L}^d$ is absolutely $(1,\mathcal{H}_M)$-decaying in $(\R^d, \dist)$. Hence, applying \cite[Theorem 5.5]{BFS} with $X=J=(\R^d,\dist)$, $\mu=\mathcal{L}^d$, $\eta=1$, $\delta=d$ and $\mathcal{H}=\mathcal{H}_M$, we get:
\begin{theorem}\label{Theo:BFS}
Let $S \subseteq \R^d$ be an $(\alpha, \beta, c, \rho, \mathcal{H}_M)$-potential winning set, with $c \in (0,1)$ and $\beta \in (0, \frac{1}{4}]$. Then, for every ball $B$ of radius at least $\rho$, we have
\[
\dim_{H}(S \cap B) \geq d -K_1 \frac{\alpha}{|\log (\beta)|}>0 \text{ if } \alpha^c \leq \frac{1}{K_2}(1-\beta^{1-c}),
\]
where $K_1$ and $K_2$ are large constants independent of $\alpha$, $\beta$, $c$, $\rho$ (but possibly depending on $d$, $\dist$, $M$).
\end{theorem}

From now on, $K_1$ and $K_2$ will be the constants given by the previous Theorem.

\begin{corollary}\label{Theo:Intersections}
Let $(C_i)_i \subseteq (\R^d, \dist)$ be a family of countably many compact sets as in Proposition \ref{prop:compact winning} for the same $(\R^d, \dist)$, $N_0$ and $\kappa$, so that
\begin{enumerate}[\rm(i)]
\item\label{H1} $R:=\sup_i \rad (S^i_{\emptyset})< \infty$,
\item\label{H2} there is a ball $B \subseteq \bigcap_i S_{\emptyset}^i$,
\item\label{H3} $\sup_{i,n} \frac{r^i_{n+1}}{r^i_n} \leq \beta_0:= \min\{ \frac{1}{4}, \frac{\rad(B)}{R}\}$,
\item\label{H4} there is $c_0 \in (0,1)$ so that $\sum_i \tau_i^{-c_0} \leq \frac{1}{K_2}(1- \beta_0^{1-c_0})$.
\end{enumerate}
Then,
\[
\dim_H \left(\bigcap_i C_i \right)\geq d -K_1 \frac{\left( \sum_i \tau_i^{-c_0}\right)^{\frac{1}{c_0}}}{\beta_0 |\log (\beta_0)|} >0.
\]
\end{corollary}

\begin{proof}
Let $s=\sup_{i, n} \frac{r^i_{n+1}}{r^i_n}$. It follows from Proposition \ref{prop:compact winning} and Lemma \ref{Monotonicity} that $S_i:=C_i \cup (\R^d\setminus S_{\emptyset}^i)$ is $\big(\frac{1}{\tau_i}, \beta, c, \beta R, \mathcal{H}_M\big)$-winning for every $c > 0$ and $\beta \in [s,\frac{1}{4}]$, where $M$ is independent of $i$.

Note that $\beta_0\in [s, \frac{1}{4}]$ by \eqref{H3}. By Lemma \ref{Countable intersection property}, $S:=\bigcap_i S_i$ is $\big(\alpha, \beta_0, c_0, \beta_0 R, \mathcal{H}_M\big)$-winning
for $\alpha= \left( \sum_i \tau_i^{-c_0}\right)^\frac{1}{c_0}$ where, by \eqref{H4},
\[
\alpha^{c_0}= \sum_i \tau_i^{-c_0}\leq \frac{1}{K_2}(1- \beta_0^{1-c_0}).
\]
Then, since $\rad (B) \geq \beta_0 R$, we conclude from Theorem \ref{Theo:BFS} that
\[
\dim_{H}(S \cap B)\geq d -K_1 \frac{\left( \sum_i \tau_i^{-c_0}\right)^{\frac{1}{c_0}}}{\beta_0 |\log (\beta_0)|}>0.
\]
Since, by \eqref{H2}, $S_i \cap B = C_i \cap B$ for every $i$, we see that $S \cap B \subseteq \bigcap_i C_i$, and the conclusion follows.
\end{proof}

\subsection{Application: patterns in thick compact sets}

As a final application, we deduce:
\begin{theorem} \label{thm:patterns}
Let $C \subseteq \R^d$ be a compact set as in Proposition \ref{prop:compact winning} with thickness $\tau$, where $\sup_n \frac{r_{n+1}}{r_n}\leq \frac{1}{4}$.
Then, $C$ contains a homothetic copy of every set with at most
\[
N(\tau):=\left \lfloor \frac{3}{4 e K_2} \frac{\tau}{\log \tau} \right \rfloor
\]
elements. Moreover, if $A$ is such a set, then for all $\lambda \in (0, \frac{3\rad (S_\emptyset)}{4 \diam (A)})$, there exists a set $X$ of positive Hausdorff dimension (depending on $A$, $C$, $S_\emptyset$ and $\lambda$) such that
\[
x+ \lambda A \subseteq C \text{ for all } x \in X.
\]
\end{theorem}

We make some remarks on this theorem. It is a well known consequence of the Lebesgue density theorem that sets of positive Lebesgue measure contain homothetic copies of every finite set. Theorem \ref{thm:patterns} shows that thick sets of measure zero also contain homothetic copies of finite sets, up to certain size depending on the thickness. The case $d=1$ was established in \cite{Y21}. This was extended to higher dimensions in \cite{FalconerYavicoli} for the different notion of thickness considered there. In that paper, the allowed size of the set is $\sim \tfrac{\tau^d}{\log\tau}$. Roughly speaking, this is because in the potential game the set $\mathcal{H}$ considered in \cite{FalconerTechniques} consists of points while here we have to take unions of spheres (the exponent of $\tau$ is the co-dimension of the sets in $\mathcal{H}$). On the other hand, Theorem \ref{thm:patterns} shows that some totally disconnected sets in arbitrary dimension contain large homothetic patterns. The value of $K_2$ is effective in principle (but small).

Theorem \ref{thm:patterns} fits into a large and expanding literature about the presence (or absence) of geometric patterns in fractal sets. See \cite{FGP, KOS, LiangPramanik} and references there for some (different) recent results in this area.

\begin{proof}[Proof of Theorem \ref{thm:patterns}]
%We take $B_0:=\frac{1}{8} S_\emptyset$ the ball with the same center as $S_\emptyset$ but radius $\frac{1}{8} \rad(S_\emptyset)$.
Given a finite set $A$ and $\lambda \in (0, \frac{3\rad (S_\emptyset)}{4 \diam (A)})$, we seek translates of $\lambda A:=\{b_1, \cdots, b_N\}$ inside $C$. We may assume $b_1=0$, and so $\lambda A \subseteq B(0,\frac{3\rad (S_\emptyset)}{4})$.

We define $C_i:=C-b_i$, which is a compact set with thickness $\tau$ for every $i$ (for the translated systems of balls).
There is a ball $B \subseteq \bigcap_{1\leq i \leq N} (S_\emptyset-b_i) \subseteq S_\emptyset$ of radius $(1-\frac{3}{4})\rad(S_\emptyset)=\frac{1}{4}\rad(S_\emptyset)$.

Let $c= 1- 1/\log(\tau)$. By Corollary \ref{Theo:Intersections}, if % $N \tau^{-c} \leq \frac{1}{K_2}(1-\beta^{1-c})$ or, equivalently,
\[
N \leq \frac{1}{K_2}\tau^c (1-4^{c-1}) = \frac{1}{e K_2} \frac{\tau}{\log \tau} [(\log\tau) (1- 4^{-1/\log \tau})],
\]
then $\dim_H \left(B \cap \bigcap_i C_i \right) >0$. It can be checked that $(\log\tau) (1- 4^{-1/\log \tau})$ is increasing so (assuming $\tau\ge e$) this conclusion holds if
\[
N \leq  \frac{3}{4 e K_2} \frac{\tau}{\log \tau}.
\]

Finally, if $x \in X:=B \cap \bigcap_i C_i$, then $x+b_i \in C_i+b_i=C$ for every $1 \leq i \leq N$, hence $x+\lambda A=x+\{b_1, \cdots, b_N\} \subseteq C$ as required.
\end{proof}

%%%%%%%%%%%%%%%%%%%%%%%%%%%%%%%%%%            %%%%%%%%%%%%%%%%%%%%%%%%%%%%%%%
%%%%%%%%%%%%%%%%%%%%%%%%%%%%%%%%%%   BIBLIO   %%%%%%%%%%%%%%%%%%%%%%%%%%%%%%%
%%%%%%%%%%%%%%%%%%%%%%%%%%%%%%%%%%            %%%%%%%%%%%%%%%%%%%%%%%%%%%%%%%

\end{document}